\documentclass[10pt,a4paper]{article}

\usepackage{latexsym,amsfonts,amsmath,amssymb,amsthm,mathrsfs,color,graphicx}

\newtheorem{theorem}{Theorem}
\newtheorem{lemma}{Lemma}

\newtheorem{remark}{Remark}

\newtheorem{definition}{Definition}

\newcommand{\rd}{\, \mathrm{d}}
\newcommand{\bszero}{\boldsymbol{0}}

\newcommand{\bse}{\boldsymbol{e}}

\newcommand{\bsx}{\boldsymbol{x}}
\newcommand{\bsy}{\boldsymbol{y}}
\newcommand{\bsw}{\boldsymbol{w}}
\newcommand{\bsz}{\boldsymbol{z}}
\newcommand{\bskappa}{\boldsymbol{\kappa}}
\newcommand{\bssigma}{\boldsymbol{\sigma}}
\newcommand{\bsxi}{\boldsymbol{\xi}}
\newcommand{\FF}{\mathbb{F}}
\newcommand{\NN}{\mathbb{N}}
\newcommand{\RR}{\mathbb{R}}
\newcommand{\ZZ}{\mathbb{Z}}
\newcommand{\Scal}{\mathcal{S}}

\newcommand{\wal}{\mathrm{wal}}
\newcommand{\wor}{\mathrm{wor}}

\allowdisplaybreaks

\begin{document}

\title{Quasi-Monte Carlo integration for twice differentiable functions over a triangle\thanks{
The research of T. Goda was supported by JSPS Grant-in-Aid for Young Scientists No.15K20964.
The research of K. Suzuki and T. Yoshiki was partially supported under the Australian Research Councils Discovery Projects funding scheme (project number DP150101770).
The research of K. Suzuki was partially supported under CREST, JST.}}

\author{Takashi Goda\thanks{Graduate School of Engineering, The University of Tokyo, 7-3-1 Hongo, Bunkyo-ku, Tokyo 113-8656, Japan (\tt{goda@frcer.t.u-tokyo.ac.jp})},
Kosuke Suzuki\thanks{Graduate School of Science,  Hiroshima University, 1-3-1 Kagamiyama, Higashihiroshima 739-8526, Japan ({\tt  kosuke-suzuki@hiroshima-u.ac.jp})},
Takehito Yoshiki\thanks{Research Support Department, University Management Division, Osaka City University, 3-3-138 Sugimoto, Sumiyoshi-ku, Osaka-shi, 558-8585 Japan. ({\tt tttyoshiki@gmail.com})}}

\date{\today}

\maketitle

\begin{abstract}
We study quasi-Monte Carlo integration for twice differentiable functions defined over a triangle. We provide an explicit construction of infinite sequences of points including one by Basu and Owen (2015) as a special case, which achieves the integration error of order $N^{-1}(\log N)^3$ for any $N\geq 2$. Since a lower bound of order $N^{-1}$ on the integration error holds for any linear quadrature rule, the upper bound we obtain is best possible apart from the $\log N$ factor. The major ingredient in our proof of the upper bound is the dyadic Walsh analysis of twice differentiable functions over a triangle under a suitable recursive partitioning.
\end{abstract}

\section{Introduction}
In this paper we study numerical integration of twice differentiable functions defined over a triangle $T\subset \RR^2$.
For an integrable function $f\colon T\to \RR$, we denote the true normalized integral of $f$ by
\begin{align*}
I(f) = \frac{1}{|T|}\int_{T}f(\bsx)\rd \bsx,
\end{align*}
where $|T|$ denotes the Lebesgue measure of $T$.
As an approximation of $I(f)$, we consider a linear algorithm of the form
\begin{align*}
I(f;P_N,W_N) = \sum_{n=0}^{N-1}w_nf(\bsx_n),
\end{align*}
for an $N$-element point set $P_N=\{\bsx_0,\ldots,\bsx_{N-1}\}\subset T$ and a set of real-valued weights $W_N=\{w_0,\ldots,w_{N-1}\}$.
In particular, a quasi-Monte Carlo (QMC) integration is an equal-weight quadrature rule where the weights sum up to 1, i.e., a linear algorithm with the special choice $w_n=1/N$ for all $n$.
Therefore, $I(f)$ is simply approximated by
\begin{align*}
I(f;P_N) = \frac{1}{N}\sum_{n=0}^{N-1}f(\bsx_n).
\end{align*}
If an infinite sequence of points $\Scal=\{\bsx_n\in T\mid n\geq 0\}$ is given, the first $N$ elements of $\Scal$ are used as $P_N$.

We define the norm in $C^2(T)$ by 
\begin{align*}
\|f\|_{C^2(T)} := \max_{0\leq \delta_1+\delta_2\leq 2}\left\| \frac{\partial^{\delta_1+\delta_2}f}{\partial x_1^{\delta_1}\partial x_2^{\delta_2}}\right\|_{L^{\infty}(T)},
\end{align*}
and study the worst-case absolute error over the unit ball of $C^2(T)$, i.e.,
\begin{align*}
e^{\wor}(C^2(T);P_N) = \sup_{\substack{f\in C^2(T)\\ \|f\|_{C^2(T)}\leq 1}}|I(f;P_N)-I(f)|.
\end{align*}
Thus an obvious goal in this context is to construct a good point set or sequence in $T$ such that the quantity $e^{\wor}(C^2(T);P_N)$ is small either for some $N$ or uniformly for all $N\geq 2$.

The theory of QMC integration has been developed in depth with the particular focus on approximating the integral of functions defined over the unit cube $[0,1]^s$, see for instance \cite{DPbook,Nbook,SJbook}.
In fact, not much attention has been paid to QMC integration over non-cubical domains until recently.
We have to point out, however, that many practical problems are not necessarily given by quadrature over the unit cube.
So far, the most standard approach to QMC integration over a non-cubical domain $\Omega$ is to find a uniformity-preserving transformation $g\colon [0,1]^s\to \Omega$ and then to approximate the normalized integral of $f\colon \Omega\to \RR$ by 
\begin{align*}
\frac{1}{N}\sum_{n=0}^{N-1}f\circ g(\bsx_n),
\end{align*}
for $\bsx_0,\ldots,\bsx_{N-1}\in [0,1]^s$.
In the literature, Fang and Wang \cite{FWbook} introduced several transformations from the unit cube to the ball, sphere, and simplex.
Pillards and Cools \cite{PC05} studied 5 different transformations from the unit cube to the simplex.
More recently, Basu and Owen \cite{BO16} gave sufficient conditions on $g$ so that $f\circ g$ is either of bounded variation or satisfies additional smoothness conditions.

Instead of applying a uniformity-preserving transformation, more direct and explicit constructions of point sets and sequences in a triangular domain $T$ have been introduced recently by Basu and Owen \cite{BO15}.
One is based on the van der Corput sequence in base 4 in conjunction with a recursive partitioning of $T$.
The other is given by a rotation of an integer lattice through an angle whose tangent is badly approximable.
A discrepancy measure derived in \cite{BCGT13} was employed as a quality criterion of these constructions, and it was shown that the latter one attains a lower discrepancy.
Nonetheless, the former one is of practical importance since it is extensible and can be randomized.

In this paper, motivated by the first construction of Basu and Owen, we study QMC integration for smooth functions in $C^2(T)$.
In particular, we give an explicit construction of infinite sequences of points including one by Basu and Owen as a special case, and prove that our quadrature rule achieves the worst-case error of order $N^{-1}(\log N)^3$ for any $N\geq 2$ in $C^2(T)$.
The main result of this paper can be summarized as follows:
\begin{theorem}\label{thm:main}
For a triangle $T\subset \RR^2$, we can explicitly construct an infinite sequence $\Scal$ of points in $T$ for which there exists a constant $C>0$ such that
\begin{align*}
e^{\wor}(C^2(T);P_N) \leq C\frac{(\log_2 N)^3}{N},
\end{align*}
for all $N\geq 2$, and in particular,
\begin{align*}
e^{\wor}(C^2(T);P_{2^m}) \leq C \frac{m^2}{2^m},
\end{align*}
for all $m\in \NN$.
\end{theorem}
\noindent Roughly speaking, our approach for the proof of Theorem~\ref{thm:main} is to exploit the decay of the Walsh coefficients for $f\in C^2(T)$ under a suitable recursive partitioning of $T$.
By following the essentially same argument using bump functions as in \cite{B59}, see also \cite[Section~2.7]{DHP15}, we see that a lower bound of order $N^{-1}$ on the worst-case error holds for any linear algorithm in $C^2(T)$.
Namely, there exists a constant $c>0$ such that
\begin{align}\label{eq:lower_bound}
e^{\wor}(C^2(T);P_N,W_N) := \sup_{\substack{f\in C^2(T)\\ \|f\|_{C^2(T)}\leq 1}}|I(f;P_N,W_N)-I(f)|\geq \frac{c}{N},
\end{align}
holds for any choice of $P_N$ and $W_N$.
Thus the upper bound we obtain is best possible apart from the $\log N$ factor.

The rest of this paper is organized as follows.
In the next section we present an explicit construction of infinite sequences of points in $T$.
We prove an upper bound on the worst-case error for our quadrature rule in Section~\ref{sec:upper}, whereas the result on the decay of the Walsh coefficients for $f\in C^2(T)$, which is necessary for the proof of an upper bound, is shown later in Section~\ref{sec:walsh}.

Throughout this paper we use the following notations.
Let $\ZZ$ be the set of integers, $\NN$ the set of positive integers, and $\NN_0:=\NN \cup \{0\}$.
We denote the two-element field by $\FF_2$, which is identified with the set $\{0,1\}\subset \ZZ$ equipped with addition and multiplication modulo 2.
The addition operation in $\FF_2$ is denoted by $\oplus$, and in case of vectors or matrices over $\FF_2$, $\oplus$ is applied componentwise.
Further we denote a triangular domain with vertices $A,B,C\in \RR^2$ by
\begin{align*}
\triangle(A,B,C) := \{w_1A+w_2B+w_3C \mid w_1,w_2,w_3\geq 0, w_1+w_2+w_3=1\} ,
\end{align*}
and the diameter of a set $S\subset \RR^2$ by $d(S)$.
Without loss of generality, we assume that the center of a triangle $T$ is located at the origin in $\RR^2$, i.e., if the center of $T$ is not located at the origin, it suffices to shift the whole domain $T$.

\section{Explicit construction}\label{sec:explicit}

\subsection{Recursive partitioning}\label{subsec:partition}
In a similar way to \cite[Section~3]{BO15}, here we introduce a recursive partitioning of a triangle $T=\triangle(A,B,C)$.
We first partition the triangle into 4 congruent subtriangles, to each of which a pair $(\xi_{11},\xi_{12})\in \FF_2^2$ is assigned with $(0,0)$ in the center.
Then we partition each subtriangle into 4 congruent sub-subtriangles, to each of which a pair $(\xi_{21},\xi_{22})\in \FF_2^2$ is assigned again with $(0,0)$ in the center.
Hence every sub-subtriangle can be now identified with a set of pairs $(\xi_{11},\xi_{12})$ and $(\xi_{21},\xi_{22})$.
This is illustrated in Figure~\ref{fig:partition}.
It is obvious that this recursive partitioning of the triangle defines the mapping from $\FF_2^{\NN\times 2}$ to $T$, which is surjective but not injective.
Moreover, for a matrix $X=(\xi_{ij})\in \FF_2^{\NN\times 2}$ with $\xi_{ij}=0$ for all $i>n$, the first $n$ rows of $X$ determines which of $4^n$ congruent subregions the matrix $X$ is mapped within, and from the condition that the pair $(0,0)$ is always assigned in the center, we see that the matrix $X$ is mapped to the center of the corresponding subregion.
\begin{figure}
\begin{center}
\includegraphics[width=9cm]{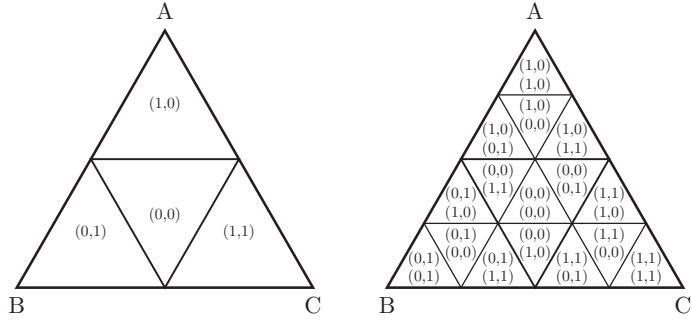}
\caption{Recursive partitioning of the triangle $T=\triangle(A,B,C)$.}
\label{fig:partition}
\end{center}
\end{figure}

We now describe our recursive partitioning more precisely.
The subtriangle of $T=\triangle(A,B,C)$ for a pair $(\xi_{11},\xi_{12})\in \FF_2^2$ is defined by
\begin{align*}
T^{(1)}(\xi_{11},\xi_{12}) =
\begin{cases}
\triangle(\frac{B+C}{2}, \frac{C+A}{2}, \frac{A+B}{2}) & (\xi_{11},\xi_{12})=(0,0), \\ 
\triangle(A, \frac{A+B}{2}, \frac{A+C}{2}) & (\xi_{11},\xi_{12})=(1,0), \\
\triangle(\frac{B+A}{2}, B, \frac{B+C}{2}) & (\xi_{11},\xi_{12})=(0,1), \\
\triangle(\frac{C+A}{2}, \frac{C+B}{2}, C) & (\xi_{11},\xi_{12})=(1,1).
\end{cases}
\end{align*}
Then the sub-subtriangle for a set of pairs $(\xi_{11},\xi_{12})$ and $(\xi_{21},\xi_{22})$ is defined by
\begin{align*}
T^{(2)}\left(\begin{matrix} \xi_{11} & \xi_{12} \\ \xi_{21} & \xi_{22}\end{matrix}\right) = \left(T^{(1)}(\xi_{11},\xi_{12})\right)^{(1)}(\xi_{21},\xi_{22}).
\end{align*}
In this way, the subregion for a matrix $X=(\xi_{ij})_{1\leq i\leq n, j=1,2}\in \FF_2^{n\times 2}$ with some $n\in \NN$ is defined recursively by
\begin{align*}
T^{(n)}\left(\begin{matrix} \xi_{11} & \xi_{12} \\ \vdots & \vdots \\ \xi_{n1} & \xi_{n2} \end{matrix}\right) & = \left(T^{(n-1)}\left(\begin{matrix} \xi_{11} & \xi_{12} \\ \vdots & \vdots \\ \xi_{n-1,1} & \xi_{n-1,2} \end{matrix}\right)\right)^{(1)}(\xi_{n1},\xi_{n2}) \\
& = \left(T^{(1)} \cdots \left(T^{(1)}(\xi_{11},\xi_{12})\right)^{(1)}\cdots \right)^{(1)}(\xi_{n1},\xi_{n2}).
\end{align*}
For simplicity of notation, as long as $1\leq i\leq n$, we write
\begin{align*}
T^{(i)}\left(\begin{matrix} \xi_{11} & \xi_{12} \\ \vdots & \vdots \\ \xi_{n1} & \xi_{n2} \end{matrix}\right) = T^{(i)}\left(\begin{matrix} \xi_{11} & \xi_{12} \\ \vdots & \vdots \\ \xi_{i1} & \xi_{i2} \end{matrix}\right).
\end{align*}
Moreover we define the mapping $\phi^{(n)}\colon \FF_2^{n\times 2}\to T$ by
\begin{align*}
\phi^{(n)}\colon X\in \FF_2^{n\times 2} \to \text{the center of the subregion $T^{(n)}(X)$} ,
\end{align*}
and, again for simplicity of notation, as long as $1\leq i \leq n$ we write
\begin{align*}
\phi^{(i)}\left(\begin{matrix} \xi_{11} & \xi_{12} \\ \vdots & \vdots \\ \xi_{n1} & \xi_{n2} \end{matrix}\right) = \phi^{(i)}\left(\begin{matrix} \xi_{11} & \xi_{12} \\ \vdots & \vdots \\ \xi_{i1} & \xi_{i2} \end{matrix}\right).
\end{align*}

Regarding the map $\phi^{(n)}$ we have the following lemma.
Since the result can be easily proved by induction on $i$, we omit the proof.
\begin{lemma}\label{lem:partition}
For a matrix $X=(\xi_{ij})_{1\leq i\leq n, j=1,2}\in \FF_2^{n\times 2}$, we define $\eta_i(X) \in \{\pm 1\}$ by $\eta_1(X):=1$ and
\begin{align*}
\eta_i(X) := (-1)^{|\{1 \leq a\leq i-1 \mid (\xi_{a1},\xi_{a2}) = (0,0) \}|},
\end{align*}
for $2 \leq i \leq n+1$.
Let $T=\triangle(\bse(1,0), \bse(0,1),\bse(1,1))$ with $\bse(1,0), \bse(0,1), \bse(1,1) \in \RR^2$ and $\bse(0,0)=(\bse(1,0)+\bse(0,1)+\bse(1,1))/3=\bszero$. 
Then the following holds true:
\begin{enumerate}
\item For $1\leq i\leq n$, we have
\begin{align*}
\phi^{(i)}(X)=\sum_{j=1}^{i}\frac{\eta_j(X)}{2^j}\bse(\xi_{j1},\xi_{j2}).
\end{align*}
\item  For $\bssigma\in \FF_2^2\setminus \{(0,0)\}$, define $\phi_{X}^{(i)}(\bssigma) \in \RR^2$ by
\begin{align*}
\phi_{X}^{(i)}(\bssigma) := \phi^{(i)}(X) + \frac{\eta_{i+1}(X)}{2^i}\bse(\bssigma).
\end{align*}
Then we have $T^{(i)}(X) = \triangle(\phi_{X}^{(i)}(1,0),\phi_{X}^{(i)}(0,1),\phi_{X}^{(i)}(1,1))$.
In particular
\begin{align*}
T^{(i)}(X) =\phi^{(i)}(X) + \frac{\eta_{i+1}(X)}{2^i}T.
\end{align*}
\end{enumerate}
\end{lemma}

\subsection{Generating infinite sequences of points in a triangle}
We describe how to generate an infinite sequence of points in a triangle $T$.
For this purpose, we first introduce the definition of digital nets over $\FF_2$ for the two-dimensional case.
\begin{definition}\label{def:digital_net}
For $m,n\in \NN$ with $n\geq m$, let $C_1,C_2\in \FF_2^{n\times m}$.
For an integer $0\leq h<2^m$, denote the dyadic expansion of $h$ by $h=\eta_0+\eta _1 2+\cdots +\eta_{m-1}2^{m-1}$.
Define the matrix $X(h)=(\xi^{(h)}_{ij})_{1\leq i\leq n, j=1,2}\in \FF_2^{n\times 2}$ by
\begin{align*}
(\xi^{(h)}_{1j},\xi^{(h)}_{2j},\ldots,\xi^{(h)}_{nj})^{\top} = C_j \cdot (\eta_0,\eta_1,\ldots,\eta_{m-1})^{\top},
\end{align*}
for $j=1,2$.
Then we call the subset $P=\{X(h) \mid 0\leq h<2^m\}\subset \FF_2^{n\times 2}$ a (two-dimensional) digital net over $\FF_2$ with generating matrices $C_1,C_2$.
\end{definition}

\begin{remark}
By using the map $\psi_n\colon \FF_2^n \to [0,1)$ defined by
\begin{align*}
\psi_n\left(\begin{matrix} \xi_1 \\ \vdots \\ \xi_n \end{matrix}\right) := \frac{\xi_1}{2}+\frac{\xi_2}{2^2}+\cdots + \frac{\xi_n}{2^n},
\end{align*}
for $(\xi_1,\ldots,\xi_n)^{\top}\in \FF_2^n$, digital nets over $\FF_2$ are usually defined as point sets in $[0,1)^2$ by $\psi_n(P):=\{\psi_n(X(h)) \mid 0\leq h<2^m\}\subset [0,1)^2$, where $\psi_n$ is applied columnwise.
Here we see that the integer $n$ denotes the precision of points.
In this paper, it is more reasonable to define digital nets over $\FF_2$ as subsets in $\FF_2^{n\times 2}$ instead of point sets in $[0,1)^2$.
\end{remark}

The above definition can be extended to digital sequences over $\FF_2$.
\begin{definition}\label{def:digital_seq}
Let $C_1,C_2\in \FF_2^{\NN\times \NN}$.
For each $C_j=(c_{kl}^{(j)})_{k,l\in \NN}$, we assume $c_{kl}^{(j)}=0$ for all sufficiently large $k$. 
For $h\in \NN_0$, denote the dyadic expansion of $h$ by $h=\eta_0+\eta _1 2+\cdots +\eta_{a-1}2^{a-1}$.
Define the matrix $X(h)=(\xi^{(h)}_{ij})_{i\in \NN, j=1,2}\in \FF_2^{\NN\times 2}$ by
\begin{align*}
(\xi^{(h)}_{1j},\xi^{(h)}_{2j},\ldots)^{\top} = C_j \cdot (\eta_0,\eta_1,\ldots,\eta_{a-1},0,0,\ldots)^{\top},
\end{align*}
for $j=1,2$.
Then we call the infinite sequence $\Scal=\{X(h) \mid h\in \NN_0\}\subset \FF_2^{\NN\times 2}$ a (two-dimensional) digital sequence over $\FF_2$ with generating matrices $C_1,C_2$.
\end{definition}

\begin{remark}\label{rem:digital_seq}
It follows from the assumption $c_{kl}^{(j)}=0$ for all sufficiently large $k$ that, for each $h\in \NN_0$, there exists a unique $\nu(h)\in \NN_0$ such that $\xi^{(h)}_{\nu(h)+1,j}=\xi^{(h)}_{\nu(h)+2,j}=\cdots =0$ for both $j=1,2$.
Furthermore, for $m\in \NN$, the first $2^m$ elements of $\Scal$ can be regarded as a digital net over $\FF_2$ with generating matrices $C_1^{n\times m},C_2^{n\times m}$ for some $n\geq m$, where we denote by $C_j^{n\times m}$ the left upper $n\times m$ sub-matrix of $C_j$.
\end{remark}

Now we are ready to present how to generate an infinite sequence of points in a triangle $T$.
\begin{definition}\label{def:digital_seq_T}
Let $\Scal\subset \FF_2^{\NN\times 2}$ be a digital sequence over $\FF_2$ with generating matrices $C_1,C_2$.
Then an infinite sequence of points in $T$ is given by
\begin{align*}
\Scal_T=\{\phi^{(\nu(h))}(X(h)) \mid h\in \NN_0\},
\end{align*}
where the function $\nu\colon \NN_0\to \NN_0$ is given as in Remark~\ref{rem:digital_seq}.
\end{definition}
\noindent
It is clear from this definition that our infinite sequence of points is determined by generating matrices $C_1,C_2$.
Thus we need some quality measure for generating matrices to make an explicit construction of $\Scal_T$ possible, which is discussed in the next subsection.

\subsection{Dual net and a new weight function}
We first recall the notion of dual net.
\begin{definition}\label{def:dual}
For $m,n\in \NN$ with $n\geq m$, let $P\subset \FF_2^{n\times 2}$ a digital net over $\FF_2$ with generating matrices $C_1,C_2\in \FF_2^{n\times m}$.
The dual net of $P$ is defined by
\begin{align*}
P^{\perp} := \{ K=(\kappa_{ij})\in \FF_2^{n\times 2}\mid C_1^{\top}\left(\begin{matrix} \kappa_{11} \\ \vdots \\ \kappa_{n1} \end{matrix}\right) \oplus C_2^{\top}\left(\begin{matrix} \kappa_{12} \\ \vdots \\ \kappa_{n2} \end{matrix}\right) = \bszero \in \FF_2^m \}.
\end{align*}
\end{definition}

\begin{remark}
Let $\Scal\subset \FF_2^{\NN\times 2}$ be a digital sequence over $\FF_2$.
As mentioned in Remark~\ref{rem:digital_seq}, the first $2^m$ elements of $\Scal$ are a digital net over $\FF_2$ with generating matrices $C_1^{n\times m},C_2^{n\times m}$ for some $n\geq m$.
Thus the above definition of the dual net still applies to such an initial finite segment of $\Scal$.
\end{remark}

The following weight function, introduced in \cite{N86} and \cite{RT97}, is well known.
\begin{definition}\label{def:NRT_weight}
For $k=(\kappa_1,\kappa_2,\ldots)^{\top}\in \FF_2^{\NN}\setminus \{\bszero\}$, where all but only a finite number of $\kappa_i$ are 0, we define
\begin{align*}
\mu_1(k) := \max \{ i\in \NN \mid \kappa_i\neq 0 \},
\end{align*}
and $\mu_1(\bszero):=0$.
In case of a matrix $K=(k_1,k_2)\in \FF_2^{\NN\times 2}$ with $k_1,k_2\in \FF_2^{\NN}$, where all but only a finite number of elements in $k_1,k_2$ are 0, we define
\begin{align*}
\mu_1(K) := \mu_1(k_1)+\mu_1(k_2).
\end{align*}

If $k$ is an element in $\FF_2^n$ for some $n\in \NN$, by considering an injection
\begin{align*}
(\kappa_1,\ldots,\kappa_n)^{\top} \to (\kappa_1,\ldots,\kappa_n, 0, 0, \ldots)^{\top},
\end{align*}
we use the same symbol $\mu_1$ to define the weight function for such $k$.
A similar abuse of notation is also done in case of a matrix $K\in \FF_2^{n\times 2}$ for finite $n$.
\end{definition}

For a digital net $P\subset \FF_2^{n\times 2}$, we define the so-called minimum weight by
\begin{align*}
\mu_1(P^{\perp}) := \min_{K\in P^{\perp}\setminus \{\bszero\}}\mu_1(K),
\end{align*}
which has been often used as a quality measure of generating matrices for QMC integration over the unit cube.
If a digital net $P$ satisfies
\begin{align*}
\mu_1(P^{\perp}) \geq m-t+1,
\end{align*}
for some $0\leq t\leq m$, we call $P$ a digital $(t,m,2)$-nets over $\FF_2$.
Furthermore, for a digital sequence $\Scal\subset \FF_2^{\NN\times 2}$, if there exists a non-negative integer $t$ such that the first $2^m$ elements of $\Scal$ are a digital $(t,m,2)$-net over $\FF_2$ for any $m>t$, we call $\Scal$ a digital $(t,2)$-sequence over $\FF_2$.

\begin{remark}\label{rem:precision}
\begin{enumerate}
\item Any digital net satisfies the above inequality for $t=m$.
In practice, we prefer a larger value of $\mu_1(P^{\perp})$ and thus equivalently a smaller value of $t$, and $t=0$ is best possible.
We refer to \cite{DPbook,Nbook} for several explicit constructions of digital nets and sequences with small $t$-value.
\item 
In what follows, we restrict ourselves to digital $(t,2)$-sequences with upper triangular generating matrices $C_1,C_2$ which satisfy $c_{kl}^{(j)}=0$ for $k>l$.
Explicit constructions of digital sequences by Sobol \cite{S67} and Tezuka \cite{T93} hold this property.
Besides, by allowing the situation $t>m$, the first $2^m$ elements of a digital $(t,2)$-sequence can be regarded as a digital $(t,m,2)$-net for any $m\in \NN$.
\end{enumerate}
\end{remark}

Now we introduce a new weight function which suits our purpose.
\begin{definition}\label{def:new_weight}
Let $n\in \NN\cup \{\infty\}$.
For a matrix $K=(k_1,k_2)\in \FF_2^{n\times 2}$ with $k_1,k_2\in \FF_2^n$, where all but only a finite number of elements in $k_1,k_2$ are 0 if $n=\infty$, we define
\begin{align*}
v(K) := \max\{\mu_1(k_1),\mu_1(k_2)\}.
\end{align*}
\end{definition}
\noindent
We can define the weight function $v$ equivalently as follows: 
For a matrix $K=(\kappa_{ij})\in \FF_2^{n\times 2}\setminus \{\bszero\}$, where all but only a finite number of $\kappa_{ij}$ are 0 if $n=\infty$, define
\begin{align*}
v(K) := \max\{ i\in \NN \mid (\kappa_{i1},\kappa_{i2}) \neq \bszero \},
\end{align*}
and $v(\bszero):=0$.

Similarly to $\mu_1(P^{\perp})$, we define the minimum weight of a digital net $P$ by
\begin{align*}
v(P^{\perp}) := \min_{K\in P^{\perp}\setminus \{\bszero\}}v(K) .
\end{align*}
Here we prefer a digital net $P$ with a large value of $v(P^{\perp})$.
In the following lemma, we show that a digital $(t,m,2)$-net with small $t$ is exactly what we want.

\begin{lemma}\label{lem:NRT-new-weights}
Let $P\subset \FF_2^{n\times 2}$ be a digital $(t,m,2)$-net over $\FF_2$. Then we have
\begin{align*}
v(P^{\perp}) \geq \frac{m-t+1}{2}.
\end{align*}
Moreover let $\Scal=\{X(h) \mid h\in \NN_0\}\subset \FF_2^{\NN\times 2}$ be a digital $(t,2)$-sequence over $\FF_2$. Then for any $m>t$, we have
\begin{align*}
v(\{X(h) \mid 0\leq h < 2^m\}^{\perp}) \geq \frac{m-t+1}{2}.
\end{align*}
\end{lemma}

\begin{proof}
Let $K=(k_1,k_2)\in \FF_2^{\NN\times 2}$ with $k_1,k_2\in \FF_2^{\NN}$, where all but only a finite number of elements in $k_1,k_2$ are 0.
From the definitions of $\mu_1$ and $v$, we have
\begin{align*}
v(K) = \max\{\mu_1(k_1),\mu_1(k_2)\} \geq \frac{\mu_1(k_1)+\mu_1(k_2)}{2}= \frac{\mu_1(K)}{2},
\end{align*}
which gives
\begin{align*}
v(P^{\perp}) \geq \frac{\mu_1(P^{\perp})}{2}\geq \frac{m-t+1}{2}.
\end{align*}
This proves the first statement.
The second statement directly follows from the definition of a digital $(t,2)$-sequence.
\end{proof}

Hence our explicit construction of an infinite sequence of points in $T$ is to use a digital $(t,2)$-sequence over $\FF_2$ with upper triangular generating matrices which is mapped to $T$ according to Definition~\ref{def:digital_seq_T}.
In the next section, we prove that such an infinite sequence of points in $T$ achieves the almost optimal order of convergence for smooth functions in $C^2(T)$.

Before going into the proof of an error bound, we provide another explicit construction inspired by the first construction due to Basu and Owen \cite{BO15}.
Let $C_1,C_2\in \FF_2^{\NN\times \NN}$ be given by
\begin{align}\label{eq:basu-owen-matrix}
C_1 = \left(\begin{matrix} 
1 & 0 & 0 & 0 & 0 & 0 & \cdots \\
0 & 0 & 1 & 0 & 0 & 0 & \cdots \\
0 & 0 & 0 & 0 & 1 & 0 & \cdots \\
\vdots & \vdots & \vdots & \vdots & \vdots & \vdots & \ddots 
\end{matrix}\right), 
C_2 = \left(\begin{matrix} 
0 & 1 & 0 & 0 & 0 & 0 & \cdots \\
0 & 0 & 0 & 1 & 0 & 0 & \cdots \\
0 & 0 & 0 & 0 & 0 & 1 & \cdots \\
\vdots & \vdots & \vdots & \vdots & \vdots & \vdots & \ddots 
\end{matrix}\right).
\end{align}
For $h\in \NN_0$ with finite dyadic expansion $h=\eta_0+\eta _1 2+\cdots$, we have
\begin{align*}
C_1 \left(\begin{matrix} \eta_0 \\ \eta_1 \\ \vdots \end{matrix}\right) = \left(\begin{matrix} \eta_0 \\ \eta_2 \\ \vdots \end{matrix}\right)
 \quad \text{and}\quad C_2 \left(\begin{matrix} \eta_0 \\ \eta_1 \\ \vdots \end{matrix}\right) = \left(\begin{matrix} \eta_1 \\ \eta_3 \\ \vdots \end{matrix}\right).
\end{align*}
Thus for even $m$, it is obvious that the first $2^m$ elements of $\Scal$ generated by these matrices are given by
\begin{align*}
\{X(h) \mid 0\leq h < 2^m\} = \left\{\left(\begin{matrix} X_{m/2} \\ \bszero \end{matrix}\right)\mid X_{m/2}\in \FF_2^{(m/2)\times 2}\right\}.
\end{align*}
Considering the image of the map $\phi^{(m/2)}\colon \FF_2^{(m/2)\times 2}\to T$, we can easily check that the point set in $T$ obtained in this way is the same as that of Basu and Owen.
This implies that our construction scheme includes their explicit construction as a special case.

Moreover it is easy to show that the first $2^m$ elements of $\Scal$ are actually a digital $(\lceil m/2\rceil,m,2)$-net over $\FF_2$.
It can be seen from Lemma~\ref{lem:NRT-new-weights} that the minimum weight for $v$ is bounded below by $(m+1)/4$, which can be improved as follows.
Since the result follows from direct calculation, we omit the proof.
\begin{lemma}
Let $C_1,C_2\in \FF_2^{\NN\times \NN}$ be given by (\ref{eq:basu-owen-matrix}).
For $m\in \NN$, let $P$ be a digital net over $\FF_2$ with generating matrices $C_1^{m\times m},C_2^{m\times m}$.
Then we have
\begin{align*}
v(P^{\perp}) = \begin{cases}
(m+1)/2 & \text{for odd $m$,} \\
m/2+1 & \text{for even $m$.} 
\end{cases}
\end{align*}
\end{lemma}

\section{Upper bound}\label{sec:upper}
Here we prove an upper bound on the worst-case error for our quadrature rule in $C^2(T)$ by using the result later shown in Section~\ref{sec:walsh}.

\subsection{Discretized function on a triangle}
\begin{definition}\label{def:disc_func}
For an integrable function $f\colon T\to \RR$ and $n\in \NN$, we define the $n$-th discretized function $F_n\colon \FF_2^{n\times 2}\to \RR$ by 
\begin{align*}
F_n(X)=\frac{1}{|T^{(n)}(X)|}\int_{T^{(n)}(X)}f(\bsy)\rd \bsy,
\end{align*}
for $X\in \FF_2^{n\times 2}$.
\end{definition}
\noindent
Obviously we have
\begin{align}\label{eq:integral_preserve}
\frac{1}{4^n}\sum_{X\in \FF_2^{n\times 2}}F_n(X) & = \frac{1}{|T|}\sum_{X\in \FF_2^{n\times 2}}\int_{T^{(n)}(X)}f(\bsy)\rd \bsy = I(f).
\end{align}
Moreover it can be shown that $F_n$ approximates $f$ well.

\begin{lemma}\label{lem:disc_approx}
Let $X\in \FF_2^{n\times 2}$ and $\bsy\in T^{(n)}(X)$. For any $f\in C^2(T)$, we have
\begin{align*}
|f(\bsy)-F_n(X)| \leq \frac{\sqrt{2}d(T)\| f\|_{C^2(T)}}{2^n}.
\end{align*}
\end{lemma}
\begin{proof}
From Definition~\ref{def:disc_func} we have
\begin{align*}
|f(\bsy)-F_n(X)| = \frac{1}{|T^{(n)}(X)|}\left|\int_{T^{(n)}(X)}\left(f(\bsy)-f(\bsz)\right) \rd \bsz \right| \leq \sup_{\bsz\in T^{(n)}(X)}|f(\bsy)-f(\bsz)|.
\end{align*}
Let us fix $\bsy=(y_1,y_2),\bsz=(z_1,z_2)\in T^{(n)}(X)$ and
consider the line segment $Z=\{\bsy+s(\bsz-\bsy)\mid 0\le s\le 1\}$.
Since $T^{(n)}(X)$ is a triangle, and thus is convex, the set $Z$ is included in $T^{(n)}(X)$.
Hence we have
\begin{align*}
|f(\bsy)-f(\bsz)| & = \left|\int_0^{1}\sum_{j=1}^2(z_j-y_j)\frac{\partial f}{\partial x_j}(\bsy+s(\bsz-\bsy)) \rd s\right|\\
& \le |\bsz-\bsy|\cdot \left(\sum_{j=1}^2\left|\int_0^1 \frac{\partial f}{\partial x_j}(\bsy+s(\bsz-\bsy)) \rd s\right|^2\right)^{\frac{1}{2}}\\
& \le \sqrt{2} |\bsz-\bsy| \cdot \|f\|_{C^2(T)}\\
& \le \sqrt{2}d(T^{(n)}(X))\|f\|_{C^2(T)} = \frac{\sqrt{2}d(T)\| f\|_{C^2(T)}}{2^n},
\end{align*}
which completes the proof.
\end{proof}

\subsection{Walsh functions and coefficients}
In order to exploit the smoothness of functions in $C^2(T)$, we shall conduct a discrete Walsh-Fourier analysis of the discretized function $F_n$ defined on $\FF_2^{n\times 2}$ later in Section~\ref{sec:walsh}.
Right now we just introduce the definition of Walsh functions and briefly review some basic facts so as to make the proof of the main result in the next subsection accessible.

First the Walsh functions are defined as follows.
\begin{definition}\label{def:walsh}
Let $n\in \NN$ be fixed. For a matrix $K=(\kappa_{ij})\in \FF_2^{n\times 2}$, the $K$-th Walsh function $\wal_k\colon \FF_2^{n\times 2}\to \{\pm 1\}$ is defined by
\begin{align*}
\wal_{K}(X) := (-1)^{\sum_{j=1}^{2}\sum_{i=1}^{n}\kappa_{ij}\xi_{ij}} ,
\end{align*}
for $X=(\xi_{ij})\in \FF_2^{n\times 2}$.
\end{definition}

For a function $F\colon \FF_2^{n\times 2}\to \RR$, we have the following Walsh expansion:
\begin{align*}
F(X) = \sum_{K\in \FF_2^{n\times 2}}\hat{F}(K)\wal_{K}(X),
\end{align*}
where $\hat{F}(K)$ denotes the $K$-th Walsh coefficient defined by
\begin{align*}
\hat{F}(K) := \frac{1}{4^n}\sum_{X\in \FF_2^{n\times 2}}F(X)\wal_{K}(X).
\end{align*}

The following character property holds between a digital net over $\FF_2$ and Walsh functions, see for instance \cite[Lemmas~4.2 \& 4.5]{DHP15} for the proof.
\begin{lemma}
Let $P\subset \FF_2^{n\times 2}$ be a digital net over $\FF_2$. Then we have
\begin{align*}
\sum_{X\in P}\wal_{K}(X) = \begin{cases}
|P| & \text{for $K\in P^{\perp}$,} \\
0 & \text{otherwise.}
\end{cases}
\end{align*}
\end{lemma}
\noindent
Using this lemma, for any $\sigma\in \FF_2^{n\times 2}$ we have
\begin{align}
& \frac{1}{|P|}\sum_{X\in P}F(X\oplus \sigma) - \frac{1}{4^n}\sum_{X\in \FF_2^{n\times 2}}F(X) \nonumber \\
& \quad = \frac{1}{|P|}\sum_{X\in P}\sum_{K\in \FF_2^{n\times 2}}\hat{F}(K)\wal_{K}(X\oplus \sigma) - \hat{F}(\bszero) \nonumber \\
& \quad = \sum_{K\in \FF_2^{n\times 2}}\hat{F}(K)\wal_{K}(\sigma)\frac{1}{|P|}\sum_{X\in P}\wal_{K}(X) - \hat{F}(\bszero) \nonumber \\
& \quad = \sum_{K\in P^{\perp}\setminus \{\bszero\}}\hat{F}(K)\wal_{K}(\sigma), \label{eq:error_dual}
\end{align}
where the second equality stems from the fact
\begin{align*}
\wal_{K}(X\oplus Y) = \wal_{K}(X)\wal_{K}(Y),
\end{align*}
for any $X,Y\in \FF_2^{n\times 2}$.

\subsection{Proof of the main result}
In order to prove Theorem~\ref{thm:main}, it suffices to show:
\begin{theorem}\label{thm:main2}
Let $\Scal\in \FF_2^{\NN\times 2}$ be either a digital $(t,2)$-sequence with upper triangular generating matrices or a digital sequence with generating matrices given by (\ref{eq:basu-owen-matrix}), and let $\Scal_T\subset T$ be constructed according to Definition~\ref{def:digital_seq_T}.
Denote the first $N$ elements of $\Scal_T$ by $P_N$.
For any $f\in C^2(T)$, the following holds true:
\begin{enumerate}
\item For all $N\geq 2$, we have
\begin{align*}
|I(f;P_N)-I(f)| \leq C \|f\|_{C^2(T)}\frac{(\log_2 N)^3}{N}.
\end{align*}
\item For all $m\in \NN$, we have
\begin{align*}
|I(f;P_{2^m})-I(f)| \leq C \|f\|_{C^2(T)} \frac{m^2}{2^m}.
\end{align*}
\end{enumerate}
\end{theorem}

\begin{proof}
We only prove the case where $\Scal$ is a digital $(t,2)$-sequence with upper triangular generating matrices.
The case where $\Scal$ is a digital sequence with generating matrices given by (\ref{eq:basu-owen-matrix}) can be shown in exactly the same way.

We denote the dyadic expansion of $N$ by $N=2^{a_1}+\cdots + 2^{a_r}$, where $a_1>\cdots > a_r\geq 0$.
We split the first $N$ elements of $\Scal$, denoted by $\{X(h)=(\xi^{(h)}_{ij})\in \FF_2^{\NN\times 2}\mid 0\leq h<N\}$, into $r$ non-overlapping subsets 
\begin{align*}
P^{(1)} & = \{X(h)\mid 0\leq h<2^{a_1}\},\\
P^{(2)} & = \{X(h)\mid 2^{a_1}\leq h<2^{a_1}+2^{a_2}\}, \\
& \vdots \\
P^{(r)} & = \{X(h)\mid 2^{a_1}+\cdots + 2^{a_{r-1}}\leq h<2^{a_1}+\cdots + 2^{a_r}=N\}.
\end{align*}
It is the well-known property of a digital sequence that each subset $P^{(l)}$ is given by digitally shifting a digital net $\{X(h)\mid 0\leq h<2^{a_l}\}$, see for instance \cite[Proof of Theorem~4.84]{DPbook}.
That is, there exists $\sigma_l\in \FF_2^{\NN\times 2}$ such that
\begin{align*}
P^{(l)}  = \{X(h) \oplus \sigma_l\mid 0\leq h<2^{a_l}\}.
\end{align*}
Let $n=\lceil \log_2N\rceil$.
Due to the property of upper triangular matrices, all of the elements in $P^{(1)}, \ldots, P^{(r)}$ and $\sigma_1,\ldots,\sigma_r$ can have at most the first $n$ rows different from $\bszero\in \FF_2^2$.
Thus we obtain
\begin{align}
I(f;P_N)-I(f) & = \frac{1}{N}\sum_{n=0}^{N-1}f\circ \phi^{(n)}(X(h)) - I(f) \nonumber \\
& = \frac{1}{N}\sum_{l=1}^{r}\sum_{h=2^{a_1}+\cdots + 2^{a_{l-1}}}^{2^{a_1}+\cdots + 2^{a_l}-1}f\circ \phi^{(n)}(X(h)) -I(f) \nonumber \\
& = \sum_{l=1}^{r}\frac{2^{a_l}}{N}\left(\frac{1}{2^{a_l}}\sum_{h=0}^{2^{a_l}-1}f\circ \phi^{(n)}(X(h)\oplus \sigma_l)-I(f)\right).\label{eq:proof_split}
\end{align}

For each $l=1,\ldots,r$, we write 
\begin{align*}
Q_l = \left\{ (\xi^{(h)}_{ij})_{1\leq i\leq n,j=1,2}\in \FF_2^{n\times 2}\mid 0\leq h<2^{a_l}\right\},
\end{align*}
which is a digital $(t,a_l,2)$-net with generating matrices $C_1^{n\times a_l},C_2^{n\times a_l}$, see the second item of Remark~\ref{rem:precision}.
By using (\ref{eq:integral_preserve}), (\ref{eq:error_dual}), and Lemma~\ref{lem:disc_approx} we have
\begin{align*}
& \frac{1}{2^{a_l}}\sum_{h=0}^{2^{a_l}-1}f\circ \phi^{(n)}(X(h)\oplus \sigma_l)-I(f) \\
& \quad = \frac{1}{2^{a_l}}\sum_{h=0}^{2^{a_l-1}}\left(f\circ \phi^{(n)}(X(h)\oplus \sigma_l)-F_n(X(h)\oplus \sigma_l)+F_n(X(h)\oplus \sigma_l)\right) - \hat{F}_n(\bszero)\\
& \quad \leq \frac{1}{2^{a_l}}\sum_{h=0}^{2^{a_l-1}}| f\circ \phi^{(n)}(X(h)\oplus \sigma_l)-F_n(X(h)\oplus \sigma_l)| + \frac{1}{2^{a_l}}\sum_{h=0}^{2^{a_l-1}}F_n(X(h)\oplus \sigma_l) - \hat{F}_n(\bszero)\\
& \quad \leq \frac{\sqrt{2}d(T)\| f\|_{C^2(T)}}{2^n} + \sum_{K\in Q_l^{\perp}\setminus \{\bszero\}}\hat{F}_n(K)\wal_K(\sigma_l) \\
& \quad \leq \frac{\sqrt{2}d(T)\| f\|_{C^2(T)}}{2^n} + \sum_{K\in Q_l^{\perp}\setminus \{\bszero\}}|\hat{F}_n(K)| .
\end{align*}
Let $D=\max(2\sqrt{2}d(T),4(d(T))^2)$.
Applying the result obtained in Lemma~\ref{lem:bound_hat_tri}, we have
\begin{align*}
\sum_{K\in Q_l^{\perp}\setminus \{\bszero\}}|\hat{F}_n(K)| \leq D \|f\|_{C^2(T)} \sum_{K\in Q_l^{\perp}\setminus \{\bszero\}}\frac{v(K)}{2^{2v(K)}}.
\end{align*}
The sum on the right-hand side is bounded by
\begin{align*}
\sum_{K\in Q_l^{\perp}\setminus \{\bszero\}}\frac{v(K)}{2^{2v(K)}} = \sum_{w=v(Q_l^{\perp})}^{n}\frac{w}{2^{2w}}\sum_{\substack{K\in Q_l^{\perp}\\ v(K)=w}}1 \leq \sum_{w=v(Q_l^{\perp})}^{n}\frac{w}{2^{2w}}\sum_{K\in Q_l^{\perp}\cap L(w)}1,
\end{align*}
where we write $L(w)=\{K\in \FF_2^{n\times 2}\mid v(K)\leq w\}$, which is a linear subspace of $\FF_2^{n\times 2}$.
The following obvious inclusions
\begin{align*}
L(w)\subset L(w+1), \quad Q_l^{\perp}\cap L(w)\subset L(w), \quad Q_l^{\perp}\cap L(w)\subset Q_l^{\perp}\cap L(w+1)
\end{align*}
induces the injective map
\begin{align*}
(Q_l^{\perp}\cap L(w+1))/(Q_l^{\perp}\cap L(w)) \to L(w+1)/L(w).
\end{align*}
Therefore we have
\begin{align*}
\dim(Q_l^{\perp}\cap L(w+1))-\dim(Q_l^{\perp}\cap L(w)) \leq \dim(L(w+1))-\dim(L(w)) = 2.
\end{align*}
It follows from the fact $Q_l^{\perp}\cap L(w)=\{\bszero\}$ for $w< v(Q_l^{\perp})$ that
\begin{align*}
\dim(Q_l^{\perp}\cap L(w)) \leq 2(w-v(Q_l^{\perp})+1),
\end{align*}
and thus $|Q_l^{\perp}\cap L(w)|\leq 2^{2(w-v(Q_l^{\perp})+1)}$ for $w\geq v(Q_l^{\perp})$.
Now we obtain
\begin{align*}
\sum_{K\in Q_l^{\perp}\setminus \{\bszero\}}\frac{v(K)}{2^{2v(K)}} & \leq \sum_{w=v(Q_l^{\perp})}^{n}\frac{w}{2^{2w}}|Q_l^{\perp}\cap L(w)|\leq 4 \sum_{w=v(Q_l^{\perp})}^{n}\frac{w}{2^{2v(Q_l^{\perp})}} \\
& \leq 4 \sum_{w=v(Q_l^{\perp})}^{n}\frac{n}{2^{2v(Q_l^{\perp})}} \leq \frac{4n^2}{2^{2v(Q_l^{\perp})}}.
\end{align*}
Using this bound and Lemma~\ref{lem:NRT-new-weights}, the summand in (\ref{eq:proof_split}) can be bounded by
\begin{align*}
& \frac{1}{2^{a_l}}\sum_{h=0}^{2^{a_l}-1}f\circ \phi^{(n)}(X(h)\oplus \sigma_l)-I(f) \\
& \quad \leq \frac{\sqrt{2}d(T)\| f\|_{C^2(T)}}{2^n} + D \|f\|_{C^2(T)}\frac{4n^2}{2^{2v(Q_l^{\perp})}} \\
& \quad \leq D \|f\|_{C^2(T)}\left( \frac{1}{2^n}+\frac{4n^2}{2^{a_l-t+1}}\right) \leq 2^{t+2}D \|f\|_{C^2(T)}\frac{n^2}{2^{a_l}}.
\end{align*}
Plugging this bound into (\ref{eq:proof_split}), we have
\begin{align*}
| I(f;P_N)-I(f) | & \leq \sum_{l=1}^{r}\frac{2^{a_l}}{N}\left|\frac{1}{2^{a_l}}\sum_{h=0}^{2^{a_l}-1}f\circ \phi^{(n)}(X(h)\oplus \sigma_l)-I(f)\right| \\
& \leq 2^{t+2} D\|f\|_{C^2(T)}\frac{r n^2}{N}\leq 2^{t+2} D\|f\|_{C^2(T)}\frac{n^3}{N}.
\end{align*}
Hence the result for the first item follows.
The second item follows easily by considering the case $N=2^m$, for which we have $r=1$ and $n=m$.
\end{proof}

\section{Walsh analysis on a triangle}\label{sec:walsh}
In this section, we give a bound on the Walsh coefficient $\hat{F}_n(K)$ for the $n$-th discretized function $F_n\colon \FF_2^{n\times 2}\to \RR$ for $f\in C^2(T)$.
We first present a formula between the Walsh coefficient $\hat{F}_n(K)$ and the so-called dyadic differences in Lemma~\ref{lem:dyadic-hat}.
Here the dyadic differences are defined in Definition~\ref{def:dyadic-dif}.
Note that the concept of the dyadic differences is originally introduced in \cite{Y15}, while we need to change the definition slightly so as to suit our purpose, i.e, QMC integration over a triangular domain.
Converting the dyadic differences into the usual derivatives, we have a bound on the Walsh coefficient $\hat{F}_n(K)$ for $f\in C^2(T)$ in Lemma~\ref{lem:bound_hat_tri}.

\subsection{Definitions and basic results}
Here we introduce some more definitions and show some basic but necessary results related to them.
For $1\le i\le n$, $\bskappa\in\mathbb{F}_2^2$ and $X=(\bsxi_i)_{i=1}^n\in\FF_2^{n\times 2}$ with $\bsxi_i\in \FF_2^2$, we define the operation
\begin{align*}
X \oplus_{i} \bskappa = \left(\begin{matrix} \bsxi'_1 \\ \vdots \\ \bsxi'_n \end{matrix}\right)\in \FF_2^{n\times 2} \quad \text{where}\quad \bsxi'_j=
\begin{cases}
\bsxi_j \oplus \bskappa  & \text{for $j=i$},\\
\bsxi_j & \text{for $j\neq i$}.
\end{cases}
\end{align*}
Moreover, for $\bskappa,\bskappa'\in \FF_2^2$, we define
\begin{align}\label{eq:tau-def}
\tau(\bskappa, \bskappa') =
\begin{cases}
\bse(\bskappa\oplus \bskappa')-\bse(\bskappa') & \text{if $\bskappa'\not\in\{0,\bskappa\}$,}\\
\bse(\bskappa\oplus \bskappa')+\bse(\bskappa') & \text{otherwise.}
\end{cases}
\end{align}
Regarding the group operation $\oplus_i$, we have the following.

\begin{lemma}\label{lem:shift-triangle}
Let $\bskappa\in\FF_2^2\setminus\{\bszero\}$,  $X=(\bsxi_i)_{i=1}^n\in \FF_2^{n\times 2}$, and $1\leq i\leq n$.
\begin{enumerate}
\item For $\bsxi_i\not\in\{\bszero,\bskappa\}$,
\begin{align*}
T^{(n)}(X\oplus_i\bskappa) = \frac{\eta_{i}(X)}{2^i}\tau(\bskappa,\bsxi_i) + T^{(n)}(X).
\end{align*}
\item For $\bsxi_i\in\{\bszero,\bskappa\}$,
\begin{align*}
T^{(n)}(X\oplus_i\bskappa) = 2\phi^{(i-1)}(X) + \frac{\eta_i(X)}{2^i}\tau(\bskappa,\bsxi_i) - T^{(n)}(X).
\end{align*}
\end{enumerate}
\end{lemma}

\begin{proof}
Let us consider the first item.
It follows from Lemma~\ref{lem:partition} that
\begin{align*}
T^{(n)}(X\oplus_i\bskappa) = \sum_{j=1}^n\frac{\eta_{j}(X\oplus_i\bskappa)}{2^j}\bse((X\oplus_i\bskappa)_j) +\frac{\eta_{n+1}(X\oplus_i\bskappa)}{2^n}T,
\end{align*}
where we write $X\oplus_i\bskappa =((X\oplus_i\bskappa)_j)_{j=1}^n$.
For $\bsxi_i\not\in\{\bszero,\bskappa\}$ with $\bskappa \neq \bszero$, we have $\bsxi_i\oplus \bskappa \neq \bszero$, which implies $\eta_{i+1}(X\oplus_i\bskappa)=\eta_{i+1}(X)$.
Thus, by the definition of $X\oplus_i\bskappa$, we have
\begin{equation}\label{eq:Xopluskappa}
 (X\oplus_i\bskappa)_j =
\begin{cases}
\bsxi_j & \text{for $1\leq j \leq n$ with $j\neq i$,}\\
\bsxi_i\oplus\bskappa & \text{for $j=i$,}
\end{cases}
\end{equation}
and $\eta_{j}(X\oplus_i\bskappa) = \eta_{j}(X)$ for $1\leq j\leq n+1$.
Using these facts, we have
\begin{align*}
& T^{(n)}(X\oplus_i\bskappa)\\
& = \sum_{j=1}^{i-1}\frac{\eta_{j}(X)}{2^j}\bse(\bsxi_j) + \frac{\eta_{i}(X)}{2^i}\bse(\bsxi_i \oplus_i\bskappa) + \sum_{j=i+1}^{n}\frac{\eta_{j}(X)}{2^j}\bse(\bsxi_j) + \frac{\eta_{n+1}(X)}{2^n}T \\
& = \phi^{(n)}(X)-\frac{\eta_{i}(X)}{2^i}\bse(\bsxi_i) + \frac{\eta_{i}(X)}{2^i}\bse(\bsxi_i\oplus\bskappa) + \frac{\eta_{n+1}(X)}{2^n}T \\
& = \frac{\eta_{i}(X)}{2^i}\tau(\bskappa,\bsxi_i) + T^{(n)}(X).
\end{align*}
Hence we have the result.

Let us move on to the second item.
For $\bsxi_i\in\{\bszero,\bskappa\}$ with $\bskappa \neq \bszero$, we have $\{\bsxi_i,\bsxi_i\oplus\bskappa\}=\{(0,0),\bskappa\}$, which implies $\eta_{i+1}(X\oplus_i\bskappa)=-\eta_{i+1}(X)$.
Thus, by the definition of $X\oplus_i\bskappa$, we have \eqref{eq:Xopluskappa} and
\[
\eta_{j}(X\oplus_i\bskappa) = 
\begin{cases}
\eta_{j}(X) & \text{for $1\leq j\leq i$,}\\
- \eta_{j}(X) & \text{for $i< j\leq n+1$.}
\end{cases}
\]
Using these equalities we have
\begin{align*}
& T^{(n)}(X\oplus_i\bskappa)\\
& = \sum_{j=1}^{i-1}\frac{\eta_j(X)}{2^j}\bse(\bsxi_j) + \frac{\eta_i(X)}{2^i}\bse(\bsxi_i \oplus_i\bskappa) - \sum_{j=i+1}^{n}\frac{\eta_j(X)}{2^j}\bse(\bsxi_j) - \frac{\eta_{n+1}(X)}{2^n}T \\
& = 2\phi^{(i-1)}(X) - \phi^{(n)}(X) + \frac{\eta_i(X)}{2^i}\bse(\bsxi_i) + \frac{\eta_i(X)}{2^i}\bse(\bsxi_i\oplus_i\bskappa) - \frac{\eta_{n+1}(X)}{2^n}T \\
& = 2\phi^{(i-1)}(X) + \frac{\eta_i(X)}{2^i}\tau(\bskappa,\bsxi_i)  - T^{(n)}(X).
\end{align*}
Hence we have the result.
\end{proof}

Let $X=(\bsxi_i)_{i=1}^n\in \FF_2^{n\times 2}$, $\bskappa \in \FF_2^2\setminus\{\bszero\}$ and $1\leq i\leq n$.
By abuse of notation, we define the map $\cdot\oplus_i\bskappa|_{T^{(n)}(X)}$ also for a real vector $\bsy\in T^{(n)}(X)$ by
\begin{align}
\bsy\oplus_i\bskappa|_{T^{(n)}(X)} := \begin{cases}
\bsy + 2^{-i}\eta_i(X)\tau(\bskappa,\bsxi_i) & \text{for $\bsxi_i\not\in\{\bszero,\bskappa\}$,} \\
2\phi^{(i-1)}(X) - \bsy + 2^{-i}\eta_i(X)\tau(\bskappa,\bsxi_i) & \text{for $\bsxi_i\in\{\bszero,\bskappa\}.$}
\end{cases}
\label{eq:def-shift}
\end{align}
As long as there is no risk of confusion, we simply denote it as $\bsy\oplus_i\bskappa$.
By comparing this definition with the results of Lemma~\ref{lem:shift-triangle},
it is straightforward to see that the image of the restriction of the map $\cdot \oplus_i\bskappa$ to $T^{(n)}(X)$ is $T^{(n)}(X\oplus_i\bskappa)$.
By the definition of $\bsy\oplus_i\bskappa$, this map is isometric, and thus, is a $C^1$ function.
This map has the following relationship with the group operator $\oplus_i$. 

\begin{lemma}\label{lem:shift}
For any $X\in \FF_2^{n\times 2}$, the following holds true:
\begin{enumerate}
\item For $1\le i\le n$, $\bskappa\in \FF_2^2$, and $f\in L^1(T^{(n)}(X\oplus_i\bskappa))$, we have
\begin{align*}
\int_{T^{(n)}(X\oplus_i\bskappa)}f(\bsz)\rd \bsz = \int_{T^{(n)}(X)}f(\bsy \oplus_i\bskappa) \rd \bsy.
\end{align*}
\item For $1\leq i,i'\leq n$, $\bskappa,\bskappa'\in \FF_2^2$, and $f\in L^1(T^{(n)}(X\oplus_i\bskappa\oplus_{i'}\bskappa'))$, we have
\begin{align*}
\int_{T^{(n)}(X\oplus_i\bskappa\oplus_{i'}\bskappa')}f(\bsw) \rd \bsw = \int_{T^{(n)}(X)}f((\bsy\oplus_{i}\bskappa)\oplus_{i'}\bskappa') \rd \bsy.
\end{align*}
\item For $\bsy\in T^{(n)}(X)$ and $\bskappa\in \FF_2^2$, we have
\begin{align*}
\bsy,\bsy\oplus_{i'}\bskappa \in T^{(i)}(X) \quad \text{for $1\leq i < i' \leq n$,}
\end{align*}
and
\begin{align*}
|\bsy \oplus_i\bskappa - \bsy| \leq \frac{2d(T)}{2^i} \quad \text{for $1\le i\le n$.}
\end{align*}
\end{enumerate}
\end{lemma}

\begin{proof}
Let us consider the first item.
Since $\cdot \oplus_i\bskappa \colon T^{(n)}(X) \to T^{(n)}(X\oplus_i\bskappa)$ is isometric,
using the change of variables $\bsz=\bsy\oplus_i\bskappa$, we have $\rd \bsz = \rd \bsy$.
Thus the result follows. 

The second item follows from applying the first item twice.

Finally let us consider the third item.
Since $\bsy\in T^{(n)}(X)$, we also have $\bsy\in T^{(i')}(X)\subset T^{(i)}(X)$.
As above, we have $\bsy \oplus_{i'}\bskappa \in T^{(n)}(X\oplus_{i'}\bskappa)\subset T^{(i')}(X\oplus_{i'}\bskappa)$ for $\bsy\in T^{(n)}(X)$.
Since the subregion $T^{(i)}(X\oplus_{i'}\bskappa)$ with $i<i'$ does not depend on $\bskappa$ and is identical to $T^{(i)}(X)$, we have
\begin{align*}
 \bsy\oplus_{i'}\bskappa\in T^{(i')}(X\oplus_{i'}\bskappa)\subset T^{(i)}(X\oplus_{i'}\bskappa) = T^{(i)}(X).
\end{align*}
Since we now know that $\bsy,\bsy\oplus_i \bskappa\in T^{(i-1)}(X)$ for $\bsy\in T^{(n)}(X)$, it follows that
\begin{align*}
 |\bsy\oplus_i\bskappa-\bsy| \leq d(T^{(i-1)}(X))= \frac{d(T)}{2^{i-1}},
\end{align*}
which completes the proof.
\end{proof}

Furthermore, we need the following maps $\sigma,p_1,p_2$ all from $\FF_2^2$ to $\FF_2^2$:
\begin{center} 
 \begin{tabular}{c|c|c|c}
  $\bskappa$ & $\sigma(\bskappa)$ & $p_1(\bskappa)$ & $p_2(\bskappa)$ \\ \hline 
  $(0,0)$ & $(1,1)$ & $(0,0)$ & $(1,1)$ \\
  $(0,1)$ & $(0,1)$ & $(1,1)$ & $(1,0)$ \\
  $(1,0)$ & $(1,0)$ & $(1,1)$ & $(0,1)$ \\
  $(1,1)$ & $(0,1)$ & $(1,0)$ & $(1,1)$ 
  \end{tabular}
\end{center}
\noindent
For $\bskappa\in \FF_2^2$, let $P(\bskappa)=\{p_1(\bskappa),p_2(\bskappa)\}$ and $N(\bskappa)=\FF_2^2\setminus P(\bskappa)$.
It is then trivial to see $P(\bskappa)\cap N(\bskappa)=\emptyset$ and $P(\bskappa)\cup N(\bskappa)=\FF_2^2$ for any $\bskappa\in \FF_2^2$.

Now fix $K=(\bskappa_i)_{i=1}^{n}\in \FF_2^{n\times 2}$ with $\bskappa_i\in \FF_2^2$.
We divide the set $\FF_2^{n\times 2}$ into some mutually exclusive subsets:
\begin{align*}
R_0(K) & := \left(\prod_{i=1}^{v(K)-1}N(\bskappa_i)\right)\times\left(\prod_{v(K)\leq i\leq n}\FF_2^2\right), \\
R_w(K) & := \left(\prod_{i=1}^{w-1}\FF_2^2\right)\times P(\bskappa_w)\times\left(\prod_{i=w+1}^{v(K)-1}N(\bskappa_i)\right)\times\left(\prod_{v(K)\leq i\leq n}\FF_2^2\right),
\end{align*}
for $w=1,\ldots,v(K)-1$.
The following properties obviously hold:
\begin{align*}
& |R_0(K)| = 4^n\cdot 2^{1-v(K)}, \\
& |R_w(K)| = 4^{n}\cdot 2^{w-v(K)}\quad \text{for $1\le w\le v(K)-1$,} \\
&  R_w(K) \cap R_{w'}(K) = \emptyset \quad \text{for $0\leq w< w'\leq v(K)-1$, and}\\
& \cup_{w=0}^{v(K)-1}R_w(K) = \FF_2^{n\times 2}.
\end{align*}

\subsection{Bounds on Walsh coefficients}
Using the division of $\FF_2^{n\times 2}$ by $R_w(K)$ introduced in the previous section, we consider separating the $K$-th Walsh coefficient $\hat{F}(K)$ of $F\colon \FF_2^{n\times 2}\to \RR$ into the following values $R_w\hat{F}(K)$.
\begin{definition}
Let $F\colon \FF_2^{n\times 2}\to \RR$.
For $K\in \FF_2^{n\times 2}$ and $0\leq w\leq v(K)-1$, we define the Walsh coefficient of $F$ on the subset $R_w(K)$:
\begin{align*}
R_w\hat{F}(K):= \frac{1}{4^n}\sum_{X \in R_w(K)} F(X)\wal_{K}(X).
\end{align*}
\end{definition}
\noindent Note that it is obvious to see
\begin{align*}
\hat{F}(K) & = \frac{1}{4^n}\sum_{X \in \FF_2^{n\times 2}} F(X)\wal_{K}(X) \\
& = \frac{1}{4^n}\sum_{w=0}^{v(K)-1}\sum_{X \in R_w(K)} F(X)\wal_{K}(X) = \sum_{w=0}^{v(K)-1}R_w\hat{F}(K).
\end{align*}
Thus, in order to obtain an upper bound on $\hat{F}(K)$, it suffices to show an upper bound on each $R_w\hat{F}(K)$.
For this goal, we first introduce the concept of dyadic differences.
\begin{definition}\label{def:dyadic-dif}
For a function $F\colon \FF_2^{n\times 2}\to \RR$, the $i$-th dyadic difference for $K=(\bskappa_i)_{i=1}^n\in \FF_2^{n\times 2}$ is defined by
\begin{align*}
d_{K}^{(i)}F(X) := F(X \oplus_i\sigma(\bskappa_i))+\wal_{\bskappa_i}(\sigma(\bskappa_i))F(X),
\end{align*}
for $i=1,\ldots,n$.
\end{definition}

We now show the following key equalities on $R_w\hat{F}(K)$ and dyadic differences.
\begin{lemma}\label{lem:dyadic-hat}
Let $F\colon \FF_2^{n\times 2}\to \RR$ be a function. For $K=(\bskappa_i)_{i=1}^n\in \FF_2^{n\times 2} \setminus\{\bszero\}$, the following holds true:
\begin{enumerate}
\item For $0\leq w\leq v(K)-1$, we have
\begin{align*}
R_w\hat{F}(K) = -\frac{1}{2}R_w\widehat{\left(d_{K}^{(v(K))}F\right)} (K).
\end{align*}
\item For $1\leq w\leq v(K)-1$, we have
\begin{align*}
R_w\hat{F}(K) = \frac{1}{2}\wal_{\bskappa_w}(\sigma(\bskappa_w)) R_w\widehat{\left(d_{K}^{(w)}F\right)} (K).
\end{align*}
\item For $1\leq w\leq v(K)-1$, we have
\begin{align*}
R_w\hat{F}(K) = -\frac{1}{4}\wal_{\bskappa_w}(\sigma(\bskappa_w)) R_w\widehat{\left(d_{K}^{(w)}d_{K}^{(v(K))}F\right)}(K).
\end{align*}
\end{enumerate}
\end{lemma}

\begin{proof}
Let us consider the first and second items.
Denote 
\begin{align*}
p\in\{v(K),w\}\quad \text{for $w>0$,}\quad \text{and} \quad p=v(K)\quad \text{for $w=0$.}
\end{align*}
We first show that we have
\begin{align}\label{eq:Rw-shift-inv}
\{X\oplus_p\sigma(\bskappa_p)\mid X\in R_w(K)\}=R_w(K).
\end{align}
For $1\le i\le n$ with $i\neq p$, the $i$-th components of $X\oplus_p\sigma(\bskappa_p)$ and $X$ are same, and the $p$-th component of $X\oplus_p\sigma(\bskappa_p)$ is $\bsxi_p\oplus\sigma(\bskappa_p)$ for $X=(\bsxi_i)_{i=1}^n$.
Thus we only need to show that $\bsxi_p\oplus\sigma(\bskappa_p)$ belongs to the $p$-th component of $R_w(K)$.
For $p=v(K)$, it obviously holds since the $p$-th component of $R_w(K)$ is $\FF_2^2$.
For $p=w$, the $p$-th component of $R_w(K)$ is $P(\bskappa_w)$, and thus from the property
\begin{align*}
\bskappa'\oplus\sigma(\bskappa)\in P(\bskappa) \quad \text{for $\bskappa\in \FF_2^2, \bskappa'\in P(\bskappa)$,}
\end{align*}
we see that $\bsxi_p\oplus\sigma(\bskappa_p)$ belongs to the $p$-th component of $R_w(K)$.

Using the equality (\ref{eq:Rw-shift-inv}) and from the property of Walsh functions, we have
\begin{align*}
& \sum_{X \in R_w(K)}F(X \oplus_p\sigma(\bskappa_p))\wal_{K}(X) \\
& =\sum_{X\in R_w(K)}F((X \oplus_p\sigma(\bskappa_p))\oplus_p\sigma(\bskappa_p))\wal_{K}(X\oplus_p\sigma(\bskappa_p)) \\
& = \sum_{X\in R_w(K)}F(X)\wal_{K}(X)\wal_{\bskappa_p}\sigma(\bskappa_p).
\end{align*}
Then it follows that
\begin{align*}
R_w\hat{F}(K) & = \frac{1}{4^n}\sum_{X\in R_w(K)}F(X)\wal_{K}(X) \\
& = \frac{1}{4^n}\wal_{\bskappa_p}(\sigma(\bskappa_p))\sum_{X\in R_w(K)}F( X ) \wal_{K}(X)\wal_{\bskappa_p}(\sigma(\bskappa_p)) \\
& = \frac{1}{2\cdot 4^n}\wal_{\bskappa_p}(\sigma(\bskappa_p)) \\
& \quad \times \sum_{X\in R_w(K)}\left(F(X) \wal_{K}(X)\wal_{\bskappa_p}(\sigma(\bskappa_p))
 + F(X) \wal_{K}(X)\wal_{\bskappa_p}(\sigma(\bskappa_p))\right) \\
& = \frac{1}{2\cdot 4^n}\wal_{\bskappa_p}(\sigma(\bskappa_p)) \\
& \quad \times \sum_{X\in R_w(K)}\left(F( X\oplus_p\sigma(\bskappa_p) ) \wal_{K}(X) +F(X) \wal_{K}(X)\wal_{\bskappa_p}(\sigma(\bskappa_p))\right) \\
& = \frac{1}{2\cdot 4^n}\wal_{\bskappa_p}(\sigma(\bskappa_p))\sum_{X\in R_w(K)}\left( F( X \oplus_p\sigma(\bskappa_p))+\wal_{\bskappa_p}(\sigma(\bskappa_p))F(X)\right)\wal_{K}(X) \\
& = \frac{1}{2\cdot 4^n}\wal_{\bskappa_p}(\sigma(\bskappa_p))\sum_{X\in R_w(K)} d_{K}^{(p)}F(X)\cdot \wal_{K}(X)\\
& = \frac{1}{2}\wal_{\bskappa_p}(\sigma(\bskappa_p))R_w\widehat{\left(d_{K}^{(p)}F\right)}(K),
\end{align*}
which completes the proof of the second item by putting $p=w$.
Let us consider the case $p=v(K)$.
From the definition of $v$, we have $\bskappa_{v(K)}\neq 0$ and thus $\wal_{\bskappa_{v(K)}}(\sigma(\bskappa_{v(K)}))=-1$.
Hence we have the result for the first item.

From the result for the first item, to which the result for the second item is applied with $F$ replaced by $d_{K}^{(v(K))}F$, we have
\begin{align*}
R_w\hat{F}(K) & = -\frac{1}{2}R_w\widehat{\left(d_{K}^{(v(K))}F\right)} (K) \\
& = -\frac{1}{4}\wal_{\bskappa_w}(\sigma(\bskappa_w)) R_w\widehat{\left(d_{K}^{(w)}d_{K}^{(v(K))}F\right)} (K).
\end{align*}
Hence the result for the third item follows.
\end{proof}

Converting the dyadic differences to the usual derivatives, we shall get a bound on $R_w\hat{F}(K)$ where $F$ denotes the $n$-th discretized function of $f\in C^2(T)$.
As a preparation we need the following lemma.
\begin{lemma}\label{lem:dif-deriv-lemma}
Let $\bsy, \bsz_1 ,\bsz_2 \in \RR^2$ with $\bsy, \bsy+\bsz_1, \bsy+\bsz_2, \bsy+\bsz_1+\bsz_2 \in T$.
For $f \in C^2(T)$, we have
\begin{align*}
|f(\bsy+\bsz_1)-f(\bsy)| \le \sqrt{2}\| f\|_{C^2(T)} |\bsz_1| ,
\end{align*}
and
\begin{align*}
|f(\bsy+\bsz_1+\bsz_2)-f(\bsy+\bsz_1)-f(\bsy+\bsz_2)+f(\bsy)|\le 2\| f\|_{C^2(T)} |\bsz_1| |\bsz_2|.
\end{align*}
\end{lemma}

\begin{proof}
Since $T$ is convex, we have $\{\bsy+s\bsz_1+t\bsz_2 \mid 0\leq s,t\leq 1\}\subset T$.
Following a similar argument as in the proof of Lemma~\ref{lem:disc_approx}, we can get the first inequality of this lemma.
Thus let us focus on the second one.
Again in a similar way as in the proof of Lemma~\ref{lem:disc_approx}, we have for $\bsz_1=(z_{11},z_{12})$, $\bsz_2=(z_{21},z_{22})$
\begin{align*}
&  |f(\bsy+\bsz_1+\bsz_2)-f(\bsy+\bsz_1)-f(\bsy+\bsz_2)+f(\bsy)| \\
& \quad = \left| \int_0^1 \sum_{i=1}^2 z_{1i}\frac{\partial f}{\partial x_i}(\bsy+ s\bsz_1+\bsz_2) \rd s - \int_0^1 \sum_{i=1}^2 z_{1i}\frac{\partial f}{\partial x_i}(\bsy + s\bsz_1) \rd s\right|\\
& \quad = \left| \sum_{i=1}^2 z_{1i}\left[ \int_0^1 \left(\frac{\partial f}{\partial x_i}(\bsy+ s\bsz_1+\bsz_2) - \frac{\partial f}{\partial x_i}(\bsy + s\bsz_1) \right) \rd s\right] \right| \\
& \quad \leq |\bsz_1| \left[ \sum_{i=1}^2 \left| \int_0^1 \left(\frac{\partial f}{\partial x_i}(\bsy+ s\bsz_1+\bsz_2) - \frac{\partial f}{\partial x_i}(\bsy + s\bsz_1) \right) \rd s\right|^2 \right]^{1/2} .
\end{align*}
The summand in the last expression for a given $i$ is bounded by
\begin{align*}
& \left| \int_0^1 \left(\frac{\partial f}{\partial x_i}(\bsy+ s\bsz_1+\bsz_2) - \frac{\partial f}{\partial x_i}(\bsy + s\bsz_1) \right) \rd s\right|^2 \\
& \quad = \left| \int_0^1 \left( \sum_{j=1}^{2}z_{2j}\int_0^1 \frac{\partial^2 f}{\partial x_i \partial x_j}(\bsy+ s\bsz_1+ t\bsz_2) \rd t \right) \rd s\right|^2 \\
& \quad \leq |\bsz_2|^2\sum_{j=1}^{2}\left| \int_0^1 \int_0^1 \frac{\partial^2 f}{\partial x_i \partial x_j}(\bsy+ s\bsz_1+ t\bsz_2) \rd t\rd s\right|^2 \leq 2\|f\|^2_{C^2(T)}|\bsz_2|^2,
\end{align*}
from which the second inequality of this lemma obviously follows.
\end{proof}

Eventually we arrive at showing upper bounds on $R_w\hat{F}(K)$ and $\hat{F}$.
\begin{lemma}\label{lem:bound_hat_tri}
Let $f\in C^2(T)$ be a function and $F_n:\FF_2^{n\times 2}\to \RR$ be its $n$-th discretized function. For any $K\in \FF_2^{n\times 2}$, we have
\begin{align*}
\left|R_0\hat{F}_n(K)\right| & \leq \frac{2\sqrt{2}d(T)\|f\|_{C^2(T)}}{2^{2v(K)}}, \\
\left|R_w\hat{F}_n(K)\right| & \leq \frac{4(d(T))^2\|f\|_{C^2(T)}}{2^{2v(K)}} \quad \text{for $1\leq w\leq v(K)-1$}, \\
\left|\hat{F}_n(K)\right| & \leq \max(2\sqrt{2}d(T),4(d(T))^2) \|f\|_{C^2(T)}\frac{v(K)}{2^{2v(K)}}.
\end{align*}
\end{lemma}

\begin{proof}
First we recall that $\wal_{\bskappa_{v(K)}}(\sigma(\bskappa_{v(K)}))=-1$ holds since $\bskappa_{v(K)}\neq 0$.
Thus for any $X\in \FF_2^{n\times 2}$ we have
\begin{align*}
d_{K}^{(v(K))}F_n(X) = F_n(X \oplus_{v(K)}\sigma(\bskappa_{v(K)}))-F_n(X).
\end{align*}
We use this equality without any notice.

We now show a bound on $R_0\hat{F}_n(K)$.
From the first item of Lemma~\ref{lem:dyadic-hat} and the triangle inequality, we have
\begin{align*}
\left|R_0\hat{F}_n(K)\right| & = \frac{1}{2}\left|R_0\widehat{\left(d_{K}^{(v(K))}F_n\right)} (K)\right| = \frac{1}{2\cdot 4^n}\left| \sum_{X \in R_0(K)} d_{K}^{(v(K))}F_n(X)\wal_{K}(X)\right| \\
& \leq \frac{1}{2\cdot 4^n}\sum_{X \in R_0(K)} \left| d_{K}^{(v(K))}F_n(X)\right| \leq \frac{1}{2^{v(K)}}\sup_{X \in R_0(K)} \left| d_{K}^{(v(K))}F_n(X)\right|.
\end{align*}
From the obvious fact $|T^{(n)}( X\oplus_v\sigma(\bskappa_{v(K)}))|=|T^{(n)}(X)|$ and the first item of Lemma~\ref{lem:shift}, we have
\begin{align*}
d_{K}^{(v(K))}F_n(X) & = F_n(X \oplus_{v(K)}\sigma(\bskappa_{v(K)}))-F_n(X) \\
 & = \frac{1}{|T^{(n)}(X)|} \left(\int_{T^{(n)}( X \oplus_{v(K)}\sigma(\bskappa_{v(K)}))}f(\bsy) \rd \bsy - \int_{T^{(n)}(X)}f(\bsy) \rd \bsy \right) \\
 & = \frac{1}{|T^{(n)}(X)|} \int_{T^{(n)}( X)} \left(f(\bsy\oplus_{v(K)} \sigma(\bskappa_{v(K)}))- f(\bsy)\right) \rd \bsy \\
 & \leq \sup_{\bsy\in T^{(n)}( X)} \left| f(\bsy\oplus_{v(K)} \sigma(\bskappa_{v(K)}))- f(\bsy)\right| \\
 & \leq \sqrt{2}\| f\|_{C^2(T)} |\bsy\oplus_{v(K)} \sigma(\bskappa_{v(K)}) - \bsy| \leq \frac{2\sqrt{2} d(T)\| f\|_{C^2(T)}}{2^{v(K)}},
\end{align*}
where we use the result in Lemma~\ref{lem:dif-deriv-lemma} with $\bsz_1=\bsy\oplus_{v(K)} \sigma(\bskappa_{v(K)}) - \bsy$ in the second inequality, and then the third item of Lemma~\ref{lem:shift} in the last inequality.
Thus we obtain a bound on $R_0\hat{F}_n(K)$:
\begin{align*}
\left|R_0\hat{F}_n(K)\right| \leq \frac{1}{2^{v(K)}}\sup_{X \in R_0(K)} \left| d_{K}^{(v(K))}F_n(X)\right| \leq \frac{2\sqrt{2} d(T)\| f\|_{C^2(T)}}{2^{2v(K)}}.
\end{align*}

Next we show a bound on $R_w\hat{F}_n(K)$ for $1\leq w\leq v(K)-1$.
From the third item of Lemma~\ref{lem:dyadic-hat} and the triangle inequality, we have
\begin{align*}
\left|R_w\hat{F}_n(K)\right| & = \frac{1}{4}\left| R_w\widehat{\left(d_{K}^{(w)}d_{K}^{(v(K))}F_n\right)}(K)\right| = \frac{1}{4^{n+1}}\left| \sum_{X \in R_w(K)} d_{K}^{(w)}d_{K}^{(v(K))}F_n(X)\wal_{K}(X) \right| \\
& \leq \frac{1}{4^{n+1}}\sum_{X \in R_w(K)} \left| d_{K}^{(w)}d_{K}^{(v(K))}F_n(X) \right| \leq \frac{1}{4\cdot 2^{v(K)-w}}\sup_{X \in R_w(K)} \left| d_{K}^{(w)}d_{K}^{(v(K))}F_n(X) \right|.
\end{align*}
From the second item of Lemma~\ref{lem:shift}, we have
\begin{align*}
& d_{K}^{(w)}d_{K}^{(v(K))}F_n(X) \\
& \quad = d_{K}^{(w)}\left( F_n(X \oplus_{v(K)}\sigma(\bskappa_{v(K)}))-F_n(X) \right) \\
& \quad = F_n(X\oplus_w\sigma(\bskappa_w)\oplus_{v(K)}\sigma(\bskappa_{v(K)})) - F_n(X\oplus_w\sigma(\bskappa_w)) \\
& \quad \qquad + \wal_{\bskappa_w}(\sigma(\bskappa_w)) \left(F_n(X\oplus_v\sigma(\bskappa_{v(K)}))-F_n(X)\right) \\
& \quad = \frac{1}{|T^{(n)}(X)|} \int_{T^{(n)}(X)}\Big( f((\bsy \oplus_{v(K)}\sigma(\bskappa_{v(K)}))\oplus_w\sigma(\bskappa_w)) - f(\bsy\oplus_w\sigma(\bskappa_w)) \\
& \quad \qquad \qquad \qquad +\wal_{\bskappa_w}(\sigma(\bskappa_w))\left(f(\bsy\oplus_{v(K)}\sigma(\bskappa_{v(K)})) - f(\bsy) \right)\Big) \rd \bsy \\
& \quad \leq \sup_{\bsy \in T^{(n)}(X)}\Big| f((\bsy \oplus_{v(K)}\sigma(\bskappa_{v(K)}))\oplus_w\sigma(\bskappa_w)) - f(\bsy\oplus_w\sigma(\bskappa_w)) \\
& \quad \qquad \qquad \qquad +\wal_{\bskappa_w}(\sigma(\bskappa_w))\left(f(\bsy\oplus_{v(K)}\sigma(\bskappa_{v(K)})) - f(\bsy) \right)\Big| .
\end{align*}
In what follows, we continue with further arguments separately for the cases $\bskappa_w=\bszero$ and $\bskappa_w\neq \bszero$.

Let us consider the case $\bskappa_w\neq \bszero$.
In this case we have $\wal_{\bskappa_w}(\sigma(\bskappa_w))=-1$.
Thus we obtain
\begin{align*}
| d_{K}^{(w)}d_{K}^{(v(K))}F_n(X)| & \leq \sup_{\bsy \in T^{(n)}(X)}\Big| f((\bsy \oplus_{v(K)}\sigma(\bskappa_{v(K)}))\oplus_w\sigma(\bskappa_w)) \\
& \quad \qquad \qquad - f(\bsy\oplus_w\sigma(\bskappa_w)) - f(\bsy\oplus_{v(K)}\sigma(\bskappa_{v(K)})) + f(\bsy) \Big| .
\end{align*}
It is easy to see by definition that $\bszero, \sigma(\bskappa_w)\not\in P(\bskappa_w)$ for $\bskappa_w\neq \bszero$. Since $X=(\bsxi_i)_{i=1}^{n} \in R_w(K)$ implies 
\begin{align*}
\bsxi_w \in P(\bskappa_w) = \{p_1(\bskappa_w),p_2(\bskappa_w)\},
\end{align*}
we have $\bsxi_w \neq \bszero, \sigma(\bskappa_w)$.
Further it follows from the third item of Lemma~\ref{lem:shift} that $\bsy\oplus_{v(K)}\sigma(\bskappa_{v(K)})\in T^{(w)}(X)$.
Thus we obtain
\begin{align*}
\bsy \oplus_w \sigma(\bskappa_w) = \bsy + \frac{\eta_w(X)}{2^w}\tau(\sigma(\bskappa_w),\bsxi_w),
\end{align*}
and
\begin{align*}
(\bsy\oplus_{v(K)}\sigma(\bskappa_{v(K)}))\oplus_w\sigma(\bskappa_w) =\bsy\oplus_{v(K)}\sigma(\bskappa_{v(K)}) + \frac{\eta_w(X)}{2^w}\tau(\sigma(\bskappa_w),\bsxi_w).
\end{align*}
Comparing these equalities gives
\begin{align*}
\bsy \oplus_w \sigma(\bskappa_w) - \bsy = (\bsy\oplus_{v(K)}\sigma(\bskappa_{v(K)}))\oplus_w\sigma(\bskappa_w) - \bsy\oplus_{v(K)}\sigma(\bskappa_{v(K)}) ,
\end{align*}
from which we see that $\bsy, \bsy\oplus_w\sigma(\bskappa_w), (\bsy\oplus_{v(K)}\sigma(\bskappa_{v(K)}))\oplus_w\sigma(\bskappa_w), \bsy\oplus_{v(K)}\sigma(\bskappa_{v(K)})$ form a parallelogram.
By using the result in Lemma~\ref{lem:dif-deriv-lemma} with $\bsz_1=\bsy\oplus_w\sigma(\bskappa_w)-\bsy$ and $\bsz_2=\bsy\oplus_{v(K)}\sigma(\bskappa_{v(K)})-\bsy$ and then the third item of Lemma~\ref{lem:shift} again, we have
\begin{align*}
| d_{K}^{(w)}d_{K}^{(v(K))}F_n(X)| & \leq 2\| f\|_{C^2(T)}\sup_{\bsy \in T^{(n)}(X)} |\bsy\oplus_w\sigma(\bskappa_w)-\bsy| |\bsy\oplus_{v(K)}\sigma(\bskappa_{v(K)})-\bsy| \\
& \leq \frac{8(d(T))^2 \| f\|_{C^2(T)} }{2^{w+v(K)}}.
\end{align*}
Thus we obtain a bound on $R_w\hat{F}_n(K)$:
\begin{align*}
\left|R_w\hat{F}_n(K)\right| \leq \frac{1}{4\cdot 2^{v(K)-w}}\sup_{X \in R_w(K)} \left| d_{K}^{(w)}d_{K}^{(v(K))}F_n(X) \right| \leq \frac{2(d(T))^2 \| f\|_{C^2(T)} }{2^{2v(K)}}.
\end{align*}

Let us move onto the case $\bskappa_w= \bszero$.
In this case we have $\wal_{\bskappa_w}(\sigma(\bskappa_w))=1$.
Thus we obtain
\begin{align*}
| d_{K}^{(w)}d_{K}^{(v(K))}F_n(X)| & \leq \sup_{\bsy \in T^{(n)}(X)}\Big| f((\bsy \oplus_{v(K)}\sigma(\bskappa_{v(K)}))\oplus_w\sigma(\bskappa_w)) \\
& \quad \qquad \qquad - f(\bsy\oplus_w\sigma(\bskappa_w)) + f(\bsy\oplus_{v(K)}\sigma(\bskappa_{v(K)})) - f(\bsy) \Big| .
\end{align*}
Since $\bskappa_w= \bszero$, we see that
\begin{align*}
\bsxi_w \in P(\bskappa_w) = \{(0,0),(1,1)\} = \{\bszero,\sigma(\bskappa_w)\}.
\end{align*}
Thus, from \eqref{eq:def-shift} we obtain
\begin{align*}
\bsy \oplus_w \sigma(\bskappa_w) = 2\phi^{(w-1)}(X) + \frac{\eta_w(X)}{2^w}\tau(\bskappa,\bsxi_w) - \bsy,
\end{align*}
and
\begin{align*}
(\bsy\oplus_{v(K)}\sigma(\bskappa_{v(K)}))\oplus_w\sigma(\bskappa_w) =2\phi^{(w-1)}(X) + \frac{\eta_w(X)}{2^w}\tau(\bskappa,\bsxi_w) - \bsy\oplus_{v(K)}\sigma(\bskappa_{v(K)}).
\end{align*}
Comparing these equalities gives
\begin{align*}
\bsy \oplus_w \sigma(\bskappa_w) - \bsy\oplus_{v(K)}\sigma(\bskappa_{v(K)}) = (\bsy\oplus_{v(K)}\sigma(\bskappa_{v(K)}))\oplus_w\sigma(\bskappa_w) - \bsy ,
\end{align*}
from which we see that $\bsy, \bsy\oplus_{v(K)}\sigma(\bskappa_{v(K)}), \bsy\oplus_w\sigma(\bskappa_w), (\bsy\oplus_{v(K)}\sigma(\bskappa_{v(K)}))\oplus_w\sigma(\bskappa_w)$ form a parallelogram.
(Note that the points $\bsy$ and $(\bsy\oplus_{v(K)}\sigma(\bskappa_{v(K)}))\oplus_w\sigma(\bskappa_w)$ form not the diagonal but the edge of the parallelogram unlike in the case of $\bskappa_w\neq \bszero$.)
By using the result in Lemma~\ref{lem:dif-deriv-lemma} with $\bsz_1=(\bsy\oplus_{v(K)}\sigma(\bskappa_{v(K)}))\oplus_w\sigma(\bskappa_w)-\bsy$ and $\bsz_2=\bsy\oplus_{v(K)}\sigma(\bskappa_{v(K)})-\bsy$ and then the third item of Lemma~\ref{lem:shift} again together with the triangle inequality, we have
\begin{align*}
& | d_{K}^{(w)}d_{K}^{(v(K))}F_n(X)| \\
& \quad \leq 2\| f\|_{C^2(T)}\sup_{\bsy \in T^{(n)}(X)} |(\bsy\oplus_{v(K)}\sigma(\bskappa_{v(K)}))\oplus_w\sigma(\bskappa_w)-\bsy| |\bsy\oplus_{v(K)}\sigma(\bskappa_{v(K)})-\bsy| \\
& \quad \leq \frac{4d(T)\| f\|_{C^2(T)}}{2^{v(K)}} \sup_{\bsy \in T^{(n)}(X)} |(\bsy\oplus_{v(K)}\sigma(\bskappa_{v(K)}))\oplus_w\sigma(\bskappa_w)-\bsy| \\
& \quad \leq \frac{4d(T)\| f\|_{C^2(T)}}{2^{v(K)}} \Big(\sup_{\bsy \in T^{(n)}(X)} |(\bsy\oplus_{v(K)}\sigma(\bskappa_{v(K)}))\oplus_w\sigma(\bskappa_w)-\bsy\oplus_{v(K)}\sigma(\bskappa_{v(K)})| \\
& \quad \qquad \qquad \qquad \qquad + \sup_{\bsy \in T^{(n)}(X)} |\bsy\oplus_{v(K)}\sigma(\bskappa_{v(K)})-\bsy|\Big) \\
& \quad \leq \frac{8(d(T))^2\| f\|_{C^2(T)}}{2^{v(K)}}\left( \frac{1}{2^w}+\frac{1}{2^{v(K)}}\right) \leq \frac{16(d(T))^2\| f\|_{C^2(T)}}{2^{w+v(K)}}.
\end{align*}
Thus we obtain a bound on $R_w\hat{F}_n(K)$:
\begin{align*}
\left|R_w\hat{F}_n(K)\right| \leq \frac{1}{4\cdot 2^{v(K)-w}}\sup_{X \in R_w(K)} \left| d_{K}^{(w)}d_{K}^{(v(K))}F_n(X) \right| \leq \frac{4(d(T))^2 \| f\|_{C^2(T)} }{2^{2v(K)}}.
\end{align*}

Finally a bound on $\hat{F}_n(K)$ is given by
\begin{align*}
\left|\hat{F}_n(K)\right| & \leq \sum_{w=0}^{v(K)-1}\left|R_w\hat{F}_n(K)\right| \\
& \leq \frac{2\sqrt{2}d(T)\|f\|_{C^2(T)}}{2^{2v(K)}} + \sum_{w=1}^{v(K)-1}\frac{4(d(T))^2 \| f\|_{C^2(T)} }{2^{2v(K)}}\\
& \leq \max(2\sqrt{2}d(T),4(d(T))^2) \|f\|_{C^2(T)}\frac{v(K)}{2^{2v(K)}}.
\end{align*}
Hence we complete the proof of this lemma.
\end{proof}


\end{document}